\documentclass[12pt]{amsart}
\usepackage{amsthm}
\usepackage{amssymb}
\usepackage{amsmath}
\usepackage{color}
\theoremstyle{plain}
\newtheorem{proposition}{Proposition}
\newtheorem{theorem}[proposition]{Theorem}
\newtheorem{lemma}[proposition]{Lemma}
\newtheorem{corollary}[proposition]{Corollary}

\theoremstyle{definition}

\newtheorem{conjecture}[proposition]{Conjecture}
\theoremstyle{definition}
\newtheorem{example}[proposition]{Example}

\newtheorem{remark}[proposition]{Remark}

\numberwithin{equation}{section}

\numberwithin{proposition}{section}

\gdef\myletter{}

\let\savetheequation\theequation
\def\theequation{\savetheequation\myletter}

\usepackage{color}

\newcommand{\CC}{{\mathbb C}}

\newcommand{\RR}{{\mathbb R}}
\newcommand{\ZZ}{{\mathbb Z}}

\renewcommand{\date}{\today}

\def \hat{\widehat}

\begin{document}


\title[Convex Bodies]{\bf Pluripotential Theory and Convex Bodies}

\author{T. Bayraktar, T. Bloom, and N. Levenberg*}{\thanks{*Supported by Simons Foundation grant No. 354549}}

\address{University of Hartford, CT 06117 USA}
\email{bayraktar@hartford.edu}

\address{University of Toronto, Toronto, Ontario M5S 2E4 Canada}
\email{bloom@math.toronto.edu}

\address{Indiana University, Bloomington, IN 47405 USA}
\email{nlevenbe@indiana.edu}

\maketitle
\begin{abstract} In their seminal paper \cite{[BB]}, Berman and Boucksom exploited ideas from complex geometry to analyze asymptotics of spaces of holomorphic sections of tensor powers of certain line bundles $L$ over compact, complex manifolds as the power grows. This yielded results on weighted polynomial spaces in weighted pluripotential theory in $\CC^d$. Here, motivated from \cite{Bay}, we work in the setting of weighted pluripotential theory arising from polynomials associated to a convex body in $(\RR^+)^d$. These classes of polynomials need not occur as sections of tensor powers of a line bundle $L$ over a compact, complex manifold. We follow the approach in \cite{[BB]} to recover analogous results.

\end{abstract}

\section{Introduction}\label{sec:intro}  Motivated by probabilistic results in \cite{Bay} as well as some questions in multivariate approximation theory \cite{BosLev}, we study pluripotential-theoretic notions associated to closed subsets $K\subset \CC^d$ and weight functions $Q$ on $K$ in the following setting. Given a convex body $P\subset (\RR^+)^d$ we define finite-dimensional polynomial spaces 
$$Poly(nP):=\{p(z)=\sum_{J\in nP\cap (\ZZ^+)^d}c_J z^J: c_J \in \CC\}, \ n=1,2,...$$
associated to $P$. Here $z^J=z_1^{j_1}\cdots z_d^{j_d}$ for $J=(j_1,...,j_d)$. 
The main goal of this work is to give a self-contained presentation of some of the results and techniques of R. Berman, S. Boucksom and D. Nystrom in \cite{[BB]} and \cite{BBN}, valid in the setting of holomorphic sections of tensor powers of certain line bundles $L$ over compact, complex manifolds, for the spaces $Poly(nP)$. A key result in \cite{[BB]} relates asymptotics of ball volume ratios of spaces of holomorphic sections with an Aubin-Mabuchi type energy of appropriate pluripotential-theoretic extremal functions. Our spaces $Poly(nP)$ do not generally arise as holomorphic sections of tensor powers of a line bundle. However, many of the techniques in \cite{[BB]} and \cite{BBN} are available and we are able to modify their approach to prove the analogous key result, Theorem \ref{keycn}, and similar consequences; e.g., that  {\it asymptotically weighted $P-$Fekete arrays} and  {\it weighted $P-$optimal measures} distribute asymptotically like the Monge-Ampere measure $(dd^cV_{P,K,Q}^*)^d$ of the weighted $P-$extremal function (Corollaries \ref{asympwtdfek} and \ref{awom}). A difference with \cite{[BB]} and \cite{BBN} is that here we deduce the existence of a weighted $P-$transfinite diameter; i.e., a limit of scaled maximal weighted Vandermondes, as a consequence of Theorem \ref{keycn} (see Remark \ref{rmk42}).

In the next section, we give definitions and background for the relevant pluripotential-theoretic notions. We define Lelong classes $L_P$ and $L_{P,+}$ associated to a convex body $P\subset (\RR^+)^d$. For certain $K\subset \CC^d$ and $Q:K\to \RR$ we define a weighted $P-$extremal function $V_{P,K,Q}$; weighted $P-$transfinite diameter, and weighted $P-$optimal measures. Ball volume ratios, as defined and utilized in \cite{[BB]}, are discussed in subsection 2.5. In section \ref{enprop} we discuss the Aubin-Mabuchi type energy $\mathcal E(u,v)$ associated to a pair of functions $u,v$ in $L_{P,+}$. The differentiability of the composition of $\mathcal E$ with a projection operator, proved in section \ref{diffep}, is a key step in verifying the main result, Theorem \ref{keycn}, on ball volume ratio asymptotics. This latter is proved in section \ref{mainth}. Both sections follow arguments in \cite{[BB]}. The applications described in the previous paragraph are given in section \ref{sec:fob}, following \cite{BBN}.

\setcounter{tocdepth}{2}
\tableofcontents

\section{Background.}\label{sec:back}

\subsection {$P-$extremal functions: Results from \cite{Bay}.} Let $\RR^+=[0,\infty)$. We fix a {\it convex body $P\subset (\RR^+)^d$; i.e., $P$ is compact, convex and $P^o\not = \emptyset$.} A standard example occurs when $P$ is a non-degenerate convex polytope, i.e., the convex hull of a finite subset of $(\ZZ^+)^d$ in $(\RR^+)^d$ with nonempty interior. Associated with $P$, following \cite{Bay}, we consider the finite-dimensional polynomial spaces 
$$Poly(nP):=\{p(z)=\sum_{J\in nP\cap (\ZZ^+)^d}c_J z^J: c_J \in \CC\}$$
for $n=1,2,...$ where $z^J=z_1^{j_1}\cdots z_d^{j_d}$ for $J=(j_1,...,j_d)$. We let $d_n$ be the dimension of $Poly(nP)$. For $P=\Sigma$ where 
$$\Sigma:=\{(x_1,...,x_d)\in \RR^d: 0\leq x_i \leq 1, \ \sum_{j=1}^d x_i \leq 1\},$$ we have $Poly(n\Sigma)=\mathcal P_n$, the usual space of holomorphic polynomials of degree at most $n$ in $\CC^d$. Given $P$, there exists a minimal positive integer $A=A(P)\geq 1$ such that $P\subset A\Sigma$. Thus
$$ Poly(nP) \subset \mathcal P_{An} \ \hbox{for all} \ n.$$

Associated to $P$ we define the {\it logarithmic indicator function} 
$$H_P(z):=\sup_{J\in P} \log |z^J|:=\sup_{J\in P} \log[|z_1|^{j_1}\cdots |z_d|^{j_d}].$$
{\it Throughout this paper, we make the assumption on $P$ that} 
\begin{equation}\label{sigmainkp} \Sigma \subset kP \ \hbox{for some} \ k\in \ZZ^+.  \end{equation}
Under this hypothesis, we have
\begin{equation} \label{sigmainkp2}  H_P(z)\geq \frac{1}{k}\max_{j=1,...,d}\log^+ |z_j|. \end{equation}
We use $H_P$ to define generalizations of the Lelong classes $L(\CC^d)$, the set of all plurisubharmonic (psh) functions $u$ on $\CC^d$ with the property that $u(z) - \log |z| = 0(1), \ |z| \to \infty$, and 
$$L^+(\CC^d)=\{u\in L(\CC^d): u(z)\geq \max_{j=1,...,d} \log^+ |z_j| + C_u\}$$
where $C_u$ is a constant depending on $u$. Define
$$L_P=L_P(\CC^d):= \{u\in PSH(\CC^d): u(z)- H_P(z) =0(1), \ |z| \to \infty \},$$ and 
$$L_{P,+}=L_{P,+}(\CC^d)=\{u\in L_P(\CC^d): u(z)\geq H_P(z) + C_u\}.$$
For $p\in Poly(nP), \ n\geq 1$ we have $\frac{1}{n}\log |p|\in L_P$; also each $u\in L_{P,+}$ is bounded below in $\CC^d$. We are working on $\CC^d$ instead of $(\CC\setminus 0)^d$ as in \cite{Bay}. Note $L_{\Sigma} = L(\CC^d)$ and $L_{\Sigma,+} = L^+(\CC^d)$. 

Given $E\subset \CC^d$, the {\it $P-$extremal function of $E$} is given by $V^*_{P,E}(z):=\limsup_{\zeta \to z}V_{P,E}(\zeta)$ where
$$V_{P,E}(z):=\sup \{u(z):u\in L_P(\CC^d), \ u\leq 0 \ \hbox{on} \ E\}.$$
Next, let $K\subset \CC^d$ be closed and let $w:K\to \RR^+$ be an {\it admissible weight function on $K$:  $w$ is a nonnegative, uppersemicontinuous function with $\{z\in K:w(z)>0\}$ nonpluripolar.} Letting $Q:= -\log w$, if $K$ is unbounded, we additionally require that 
$$\liminf_{|z|\to \infty, \ z\in K} [Q(z)- H_P(z)]=+\infty.$$
Define the {\it weighted $P-$extremal function} $$V^*_{P,K,Q}(z):=\limsup_{\zeta \to z}V_{P,K,Q}(\zeta)$$ where
$$V_{P,K,Q}(z):=\sup \{u(z):u\in L_P(\CC^d), \ u\leq Q \ \hbox{on} \ K\}. $$
If $Q=0$ we simply write $V_{P,K,Q}=V_{P,K}$, consistent with the previous notation. In the case $P=\Sigma$, 
\begin{equation} \label{vkq} V_{\Sigma,K,Q}(z)=V_{K,Q}(z):= \sup \{u(z):u\in L(\CC^d), \ u\leq Q \ \hbox{on} \ K\} \end{equation}
is the usual weighed extremal function, e.g., as in Appendix B of \cite{ST}.

We recall some results in \cite{Bay}, modified for our setting of $\CC^d$ and $P\subset (\RR^+)^d$. Our hypothesis (\ref{sigmainkp}) implies Lemma 2.2 in \cite{Bay} which was used to prove a result on total mixed Monge-Amp\`ere masses and a Siciak-Zaharjuta type theorem. Let $\omega:=dd^c  \max_{j=1,...,d} \log^+|z_j|$.

\begin{proposition} \label{turgay1} Let $P_i \subset (\RR^+)^d, \ i=1,...,k, \ k\leq d,$ be convex bodies and let $u_i,v_i\in L_{P_i}\cap L^{\infty}_{loc}(\CC^d), \ i=1,...,k$ with 
$$u_i(z) \leq v_i(z)+C_i \ \hbox{for} \ z \in \CC^d, \ i=1,...,k.$$
Then
$$\int_{\CC^d} dd^cu_1\wedge \cdots \wedge dd^cu_k \wedge \omega^{d-k} \leq \int_{\CC^d} dd^cv_1\wedge \cdots \wedge dd^cv_k \wedge \omega^{d-k}.$$
In particular, if $u_i\in L_{P_i,+}, \ i=1,...,k$, then
$$\int_{\CC^d} dd^cu_1\wedge \cdots \wedge dd^cu_k \wedge \omega^{d-k}=M_k$$
where $M_k$ is a constant depending only on $k,d,P_1,...,P_k$ (independent of $u_i\in L_{P_i,+}$).

\end{proposition}

\begin{remark} \label{totalMA} The constants $M_k$ can be computed; see Section 2.1 of \cite{Bay}. Normalizing so that $\int_{\CC^d} \omega^d=1$, for any $u\in L_{P,+}$ we have
\begin{equation}\label{norm}\int_{\CC^d} (dd^cu)^d =\int_{\CC^d} (dd^c H_P)^d = d! Vol(P)=:n_d \end{equation}
where $Vol(P)$ denotes the euclidean volume of $P\subset (\RR^+)^d$.

\end{remark}

\begin{proposition} \label{turgay2} Let $P\subset (\RR^+)^d$ be a convex body, $K\subset \CC^d$ compact, and $w=e^{-Q}$ an admissible weight on $K$. Then 
$$V_{P,K,Q} =\lim_{n\to \infty} \frac{1}{n} \log \Phi_n=\lim_{n\to \infty} \frac{1}{n} \log \Phi_{n,P,K,Q}$$
pointwise on $\CC^d$ where
$$\Phi_n(z):= \sup \{|p_n(z)|: p_n\in Poly(nP),  \ \max_{\zeta \in K} |p_n(\zeta )e^{-nQ(\zeta)}|\leq 1\}.$$
Moreover, if $Q$ is continuous, i.e., $Q\in C(K)$, and $V_{P,K,Q}$ is continuous, the convergence is locally uniform on $\CC^d$.

\end{proposition}

\begin{remark} \label{locreg} Since $ P \subset A\Sigma$, we have
$$\Phi_{n,P,K,Q}\leq \Phi_{n,A\Sigma,K,Q}.$$
In particular, for $Q=0$, 
$$\lim_{n\to \infty} \frac{1}{n} \log \Phi_{n,P,K,0}=V_{P,K} \leq A\cdot \lim_{n\to \infty} \frac{1}{An} \log \Phi_{n,A\Sigma,K,0}=A\cdot V_{\Sigma,K}.$$
Thus for any $K$, 
\begin{equation}\label{locreg} V_{\Sigma,K}^*(z)=0 \ \hbox{implies}  \ V_{P,K}^*(z)=0.\end{equation} 
A compact set $K\subset \CC^d$ is {\it locally regular} if for all $z\in K$ and all balls $B(z,r):=\{w:|w-z|\leq r\}$ we have $V_{\Sigma,K\cap B(z,r)}^*(z)=0$. As examples, the closure of any  bounded open set $D\subset \CC^d$ with $C^1$ boundary is locally regular. It is known (cf., \cite{Si}, Proposition 2.16) that if $K$ is locally regular and $Q\in C(K)$ then $V_{K,Q}$ in (\ref{vkq}) is continuous. Using (\ref{locreg}) for $K\cap B(z,r)$, the same proof shows that for {\it any} convex body $P\subset (\RR^+)^d$, if $K$ is locally regular and $Q\in C(K)$ then $V_{P,K,Q}$ is continuous. \end{remark}

Following the proofs of Lemma 2.3 and Theorem 2.5 in Appendix B of \cite{ST}, we have the following.

\begin{proposition} \label{turgay3} Let $P\subset (\RR^+)^d$ be a convex body, $K\subset \CC^d$ be closed, and let $w=e^{-Q}$ be  an admissible weight on $K$. Then $S_w:=supp (dd^cV_{P,K,Q}^*)^d$ is compact and 
\begin{equation}\label{suppw}supp\bigl( (dd^cV_{P,K,Q}^*)^d\bigr)\subset \{z\in K: V_{P,K,Q}^*(z)\geq Q(z)\}.\end{equation}
Moreover, $V_{P,K,Q}^*=Q$ q.e. on $supp(dd^cV_{P,K,Q}^*)^d$, i.e., off of a pluripolar set. In particular, if $Q$ and $V_{P,K,Q}$ are continuous,
$$supp\bigl( (dd^cV_{P,K,Q})^d\bigr)\subset \{z\in K: V_{P,K,Q}(z)=Q(z)\}.$$
\end{proposition}

\begin{remark}\label{compsupp} It follows under the hypotheses of Proposition \ref{turgay3} that 
$$V_{P,K,Q}^*=V_{P,S_w,Q|_{S_w}}^*\in L_{P,+}.$$
\end{remark}

\begin{example} \label{torus} Let $P\subset (\RR^+)^d$ be a convex body and $K=T^d$, the unit $d-$torus in $\CC^d$. Then 
\begin{equation}\label{torusextr} V_{P,T^d}(z)= H_P(z)=\max_{J\in P} \log |z^J|\in L_{P,+}.\end{equation}
This is Example 2.3 in \cite{Bay}.

\end{example}

\begin{remark} The results (and proofs) of Propositions \ref{turgay1}, \ref{turgay2} and \ref{turgay3}, as well as Example \ref{torus}, are valid for $P\subset (\RR^+)^d$ a convex body; some were stated in \cite{Bay} only in the case of $P\subset \RR^d$ a non-degenerate convex polytope. An alternate proof of (\ref{torusextr}) can be found in \cite{BosLev}. Further explicit examples of weighted $P-$extremal functions and their Monge-Amp\`ere measures can be found in \cite{Bay}.

\end{remark}

The proof of Theorem 2.6 in Appendix B of \cite{ST}, which uses a domination principle (Theorem 1.11 in Appendix B of \cite{ST}), is valid to obtain the following result. 

\begin{proposition} \label{turgay4} Let $P\subset (\RR^+)^d$ be a convex body, $K\subset \CC^d$ be closed, and let $w=e^{-Q}$ be  an admissible weight on $K$. Then for $p_n\in Poly(nP)$ with $w(z)^n|p_n(z)|\leq M$ q.e. $z\in S_w$, 
\begin{equation}\label{bwestpn} |p_n(z)|\leq M\exp(nV_{P,K,Q}^*(z)), \ z\in \CC^d \end{equation}
and
$$w(z)^n|p_n(z)|\leq M\exp[n(V_{P,K,Q}^*(z)-Q(z))], \ z\in K.$$
Hence $w(z)^n|p_n(z)|\leq M$ q.e. $z\in K$.
\end{proposition}

For $K\subset \CC^d$ compact, $w=e^{- Q}$ an admissible weight function on $K$, and $\nu$ a finite measure on $K$, we say that the triple $(K,\nu,Q)$ satisfies a weighted Bernstein-Markov property if for all $p_n\in \mathcal P_n$, 
\begin{equation}\label{wtdbm}||w^np_n||_K \leq M_n ||w^np_n||_{L^2(\nu)} \ \hbox{with} \ \limsup_{n\to \infty} M_n^{1/n} =1.\end{equation}
Here, $||w^np_n||_K:=\sup_{z\in K} |w(z)^np_n(z)|$ and 
\begin{equation}\label{wtdnorm} ||w^np_n||_{L^2(\nu)}^2:=\int_K |p_n(z)|^2  w(z)^{2n} d\nu(z).\end{equation} 
For $K$ closed but unbounded, we allow $\nu$ to be locally finite. In this setting, if $\nu(K)=\infty$ we must assume the weighted $L^2-$norms in (\ref{wtdnorm}) are finite. Next, following \cite{Bay}, given $P\subset (\RR^+)^d$ a convex body, we say that a finite measure $\nu$ with support in a compact set $K$ is a Bernstein-Markov measure for the triple $(P,K,Q)$ if (\ref{wtdbm}) holds for all $p_n\in Poly(nP)$. Again for $K$ closed but unbounded, if $\nu(K)=\infty$ we must assume the weighted $L^2-$norms in (\ref{wtdnorm}) are finite. 

\begin{remark}\label{ap} Since for any $P$ there exists $A=A(P)>0$ with $Poly(nP) \subset \mathcal P_{An}$ for all $n$, if $(K,\nu,Q)$ satisfies a weighted Bernstein-Markov property, then $\nu$ is a Bernstein-Markov measure for the triple $(P,K,\tilde Q)$ where $\tilde Q = AQ$. In particular, if $\nu$ is a {\it strong Bernstein-Markov measure} for $K$; i.e., if $\nu$ is a weighted Bernstein-Markov measure for any $Q\in C(K)$, then for any such $Q$, $\nu$ is a Bernstein-Markov measure for the triple $(P,K,Q)$. 
 \end{remark}

\begin{remark}\label{torus2} In Example \ref{torus}, the monomials $z^J, \ J\in nP\cap (\ZZ^+)^d$, form an orthonormal basis for $Poly(nP)$ with respect to normalized Haar measure $\mu_T$ on $T^d$. Moreover, $\mu_T$ is a strong Bernstein-Markov measure for $T$ and hence it is a Bernstein-Markov measure for the triple $(P,T,Q)$ for any $Q\in C(T)$. \end{remark}

\noindent We refer to \cite{blpw} for a survey of Bernstein-Markov properties.

\subsection {Projection operator.} \label{projop} To emphasize the relation between the weight $Q$ and the weighted $P-$extremal function $V_{P,K,Q}^*$, we may write 
\begin{equation}
\label{pq}
\Pi(Q) =\Pi_K(Q) :=V_{P,K,Q}^*.
\end{equation}
This  operator $\Pi$ is increasing and concave: 
if $Q_1\leq Q_2$ are admissible weights on $K$, then $\Pi(Q_1)\leq \Pi(Q_2)$; and if $0\leq s \leq 1$ and $a,a'$ are admissible weights on $K$,
\begin{equation}
\label{projcon}
\Pi(sa + (1-s)a') \geq s\Pi(a) + (1-s)\Pi(a').
\end{equation}
Since $sa + (1-s)a'$ is a convex combination of $a, a'$, it is an admissible weight on $K$. Then (\ref{projcon}) follows since the right-hand-side is a competitor for the weighted $P-$extremal function on the left-hand-side. 

It follows from the definition of $\Pi$, Proposition \ref{turgay3}, and Remark \ref{compsupp} that $\Pi$ is Lipschitz on locally regular compacta. That is, if $a,b\in C(K)$ and $0\leq t\leq 1$ then on $\CC^d$, 
\begin{equation}
\label{loclip}
|\Pi(a+t(b-a)) -\Pi(a)|\leq Ct
\end{equation}
where $C=C(a,b)=\max[\sup_{D(0)}|b-a|,\sup_{D(t)}|b-a|]$. Here $D(t):=\{\Pi(a+t(b-a))=a+t(b-a)\}$. Similarly, if $u\in C(K)$, we have, for $t\in \RR$, 
\begin{equation}
\label{loclip2}
|\Pi(a+tu) -\Pi(a)|\leq C|t|
\end{equation}
where $C=C(u)=\sup_{K}|u|$. In the former case, if $K$ is unbounded, in order that $\max[\sup_{D(0)}|b-a|,\sup_{D(t)}|b-a|]$ is a {\it finite} constant which is independent of $t$, we assume that 
\begin{equation}
\label{unbhyp}
\cup_{0\leq t\leq 1} D(t) \ \hbox{is bounded and}  \ u:=b-a\in L^{\infty}(\cup_{0\leq t\leq 1} D(t)).
\end{equation}
Then (\ref{loclip}) holds. This observation will be used in the proof of Theorem \ref{keycn}. In both cases, if $K$ is compact, $C$ is finite.

Another result we will need is a comparison principle in $ L_{P,+}$; we state and prove the version we will use.

\begin{proposition} 
\label{L+comp}
Let $a_1,a_2\in  L_{P,+}$ and $b_1,b_2 \in L^+(\CC^d)$. For $M>0$, set $u_1:=a_1+Mb_1$ and $u_2:=a_2 +Mb_2$. Then
$$\int_{\{u_1 < u_2\}}(dd^cu_2)^d \leq \int_{\{u_1 < u_2\}}(dd^cu_1)^d.$$
\end{proposition}

\noindent Note that the integrand may be unbounded but each integral is finite  by Proposition \ref{turgay1}.

\begin{proof} By adding a constant to $u_1$, if necessary, we may assume $u_1\geq 0$. Then for $\epsilon >0$, we have 
$$\{(1+\epsilon)u_1 < u_2\} \subset \{u_1 < u_2\}$$ 
and $\{(1+\epsilon)u_1 < u_2\}$ is bounded. By the standard comparison theorem for locally bounded psh functions on bounded domains (cf., Theorem 3.7.1, \cite{[K]}),
\begin{equation}
\label{compus}
\int_{\{(1+\epsilon)u_1 < u_2\}}(dd^cu_2)^d \leq (1+\epsilon)^d\int_{\{(1+\epsilon)u_1 < u_2\}}(dd^cu_1)^d.
\end{equation}
Clearly
$$\bigcup_{j=1}^{\infty} \{(1+1/j)u_1 < u_2\} = \{u_1 < u_2\}$$
so applying (\ref{compus}) with $\epsilon = 1/j$, the result follows by monotone convergence upon letting $j\to \infty$.
\end{proof}

The following lemma (and corollary) will be used in subsection \ref{diffep}.

\begin{lemma} 
\label{mamass}
Let $a$ be an admissible weight on a compact set $K$ and let $u\in C^2(K)$. Then
\begin{equation}
\lim_{t\to 0} \int_{D(0)\setminus D(t)} (dd^c\Pi(a))^d  =0
\end{equation}
where $D(t)=\{\Pi(a +tu)=a +tu\}$ for $t\in \RR$.
\end{lemma}

\begin{proof} The hypothesis $u\in C^2(K)$ means that $u$ is the restriction to $K$ of a $C^2$ function (which we also denote by $u$) on $\CC^d$; clearly we can take this function to have compact support. We prove the result for $t>0$; i.e $t\to 0^+$. We can find $M>0$ sufficiently large depending on $u$ and its support so that $u+M\psi$ is psh where $\psi(z) = \frac{1}{2} \log (1+|z|^2)$. Observing that 
$$D(0)\setminus D(t)\subset S$$
where 
$$S:=\{\Pi(a+tu)< \Pi(a) +tu\}=\{\Pi(a+tu)+tM\psi < \Pi(a) + t(u+M\psi)\}$$
and
$$D(t) \cap \{\Pi(a+tu)< \Pi(a) +tu\}=\emptyset,$$
we have
$$\int_{D(0)\setminus D(t)} (dd^c\Pi(a))^d  \leq \int_S (dd^c\Pi(a))^d $$
$$\leq \int_S [dd^c(\Pi(a)+t(u+M\psi)]^d\leq \int_S [dd^c(\Pi(a+tu)+tM\psi)]^d$$
$$=\int_S [dd^c(\Pi(a+tu))]^d +0(t) =0(t).$$
Here, the inequality in the second line comes from Proposition \ref{L+comp} (with $M\to tM$).
\end{proof}

\begin{corollary} 
\label{mamassbis}
Let $a,b\in C^2(E)$ be admissible weights on a closed, unbounded set $E$. If (\ref{unbhyp}) holds then
\begin{equation}
\lim_{t\to 0} \int_{D(0)\setminus D(t)} (dd^c\Pi(a))^d  =0
\end{equation}
where $D(t)=\{\Pi(a +t(b-a))=a +t(b-a)\}$ for $0\leq t\leq 1$.
\end{corollary}

\begin{proof} First of all, $(dd^c\Pi(a))^d$ has compact support. Also, by (\ref{unbhyp}), the $P-$extremal functions $\Pi(a+t(b-a))$ for all $0\leq t\leq 1$ are independent of the values of $a,b$ outside a large ball. Thus  we may assume that $a=b$ outside a fixed ball. In other words, this  case is reduced to the case of Lemma \ref{mamass} where $u=b-a$.
\end{proof}

\begin{remark} For the remainder of this paper, $K$ will always denote a {\it compact} subset of $\CC^d$ while $E$ will be used for a closed but possibly unbounded subset.

\end{remark}

\subsection{Transfinite diameter.}

Recall $d_n$ is the dimension of $Poly(nP)$. We can write
$$Poly(nP)= \hbox{span} \{e_1,...,e_{d_n}\}$$ 
where $\{e_j(z):=z^{\alpha(j)}\}_{j=1,...,d_n}$ are the appropriate standard basis monomials. For 
points $\zeta_1,...,\zeta_{d_n}\in \CC^d$, let
\begin{equation} \label{vdm}VDM(\zeta_1,...,\zeta_{d_n}):=\det [e_i(\zeta_j)]_{i,j=1,...,d_n}  \end{equation}
$$= \det
\left[
\begin{array}{ccccc}
 e_1(\zeta_1) &e_1(\zeta_2) &\ldots  &e_1(\zeta_{d_n})\\
  \vdots  & \vdots & \ddots  & \vdots \\
e_{d_n}(\zeta_1) &e_{d_n}(\zeta_2) &\ldots  &e_{d_n}(\zeta_{d_n})
\end{array}
\right]$$
and for a compact subset $K\subset \CC^d$ let
$$V_n =V_n(K):=\max_{\zeta_1,...,\zeta_{d_n}\in K}|VDM(\zeta_1,...,\zeta_{d_n})|.$$
We will show later that the limit
\begin{equation} \label{tdlim} \delta(K):=\delta(K,P):= \lim_{n\to \infty}V_{n}^{1/l_n} \end{equation} exists where $l_n$ is the sum of the degrees of a set of these basis monomials for $ Poly(nP)$. We call $\delta(K)$ the {\it $P-$transfinite diameter} of $K$. More generally, let $w$ be an admissible weight function on
$K$. Given $\zeta_1,...,\zeta_{d_n}\in K$, let
$$W(\zeta_1,...,\zeta_{d_n}):=VDM(\zeta_1,...,\zeta_{d_n})w(\zeta_1)^{n}\cdots w(\zeta_{d_n})^{n}$$
$$= \det
\left[
\begin{array}{ccccc}
 e_1(\zeta_1) &e_1(\zeta_2) &\ldots  &e_1(\zeta_{d_n})\\
  \vdots  & \vdots & \ddots  & \vdots \\
e_{d_n}(\zeta_1) &e_{d_n}(\zeta_2) &\ldots  &e_{d_n}(\zeta_{d_n})
\end{array}
\right]\cdot w(\zeta_1)^{n}\cdots w(\zeta_{d_n})^{n}$$
be a {\it weighted Vandermonde determinant}. Let
\begin{equation} \label{wn} W_n(K):=\max_{\zeta_1,...,\zeta_{d_n}\in K}|W(\zeta_1,...,\zeta_{d_n})|
\end{equation}
and define an {\it $n-$th weighted $P-$Fekete set for $K$ and $w$} to be a set of $d_n$ points $\zeta_1,...,\zeta_ {d_n}\in K$ with the property that
$$|W(\zeta_1,...,\zeta_{d_n})|=W_n(K).$$
We also write $\delta^{w,n}(K):=W_n(K)^{1/l_n}$ and we will show, more generally, that the {\it weighted $P-$transfinite diameter} 
\begin{equation} \label{deltaw} \delta^w(K):=\delta^w(K,P):=\lim_{n\to \infty}\delta^{w,n}(K):=\lim_{n\to \infty}W_{n}(K)^{1/l_n} \end{equation}
exists. For each $n$,  if we take points $z_1^{(n)},z_2^{(n)},\cdots,z_{d_n}^{(n)}\in K$ for which 
\begin{equation}\label{wam}
 \lim_{n\to \infty}\bigl[|VDM(z_1^{(n)},\cdots,z_{d_n}^{(n)})|w(z_1^{(n)})^nw(z_2^{(n)})^n\cdots w(z_{d_n}^{(n)})^n\bigr]^{{1\over  l_n}}=\delta^w(K)
\end{equation}
-- we call these {\it asymptotically weighted $P-$Fekete arrays} -- and we let $\mu_n:= \frac{1}{d_n}\sum_{j=1}^{d_n} \delta_{z_j^{(n)}}$, one of our results, Corollary \ref{asympwtdfek}, is that 
$$
\mu_n \to  \frac{1}{n_d}(dd^c\Pi(Q))^d \ \hbox{weak}-*.
$$
(recall (\ref{norm})).

\begin{remark}\label{numbers} For $P=\Sigma$ so that $Poly(n\Sigma)=\mathcal P_n$, we have 
$$d_n(\Sigma)={d+n \choose d}=0(n^d/d!) \ \hbox{and} \ l_n(\Sigma)= \frac{d}{d+1}nd_n(\Sigma)$$
In particular, 
$$\frac{l_n (\Sigma)}{d_n(\Sigma)}=  \frac{nd}{d+1}.$$
For a general convex body $P\subset (\RR^+)^d$ with $A>0$ so that $P \subset A\Sigma$, we write
\begin{equation}\label{lndn}
l_n = f_n(d)  \frac{nd}{d+1}d_n = f_n(d) \frac{l_n (\Sigma)}{d_n(\Sigma)}d_n.
\end{equation}
We will need to know that $l_n/d_n$ divided by $l_n (\Sigma)/d_n(\Sigma)$ has a limit; i.e., that
\begin{equation}\label{key}\lim_{n\to \infty} f_n(d)=:\mathcal A=\mathcal A(P,d)\end{equation}
exists. 
It suffices to verify (\ref{key}) for $P\subset (\RR^+)^d$ a non-degenerate convex polytope. It follows from Theorem 2 of Lecture 2 in \cite{V} 
\begin{enumerate}
\item applied to $f(j_1,...,j_d)\equiv 1$ that $d_n$ is a polynomial of degree $d$ in $n$ with 
$$d_n = Vol(P)n^d+0(n^{d-1}); \ \hbox{and}$$
\item applied to $f(j_1,...,j_d)=j_1+\cdots +j_d$ that $l_n$ is a polynomial of degree $d+1$ in $n$ with 
$$l_n = C_Pn^{d+1}+0(n^{d})$$
where $C_P=\int_P (x_1+\cdots + x_d)dx_1 \cdots dx_d$.
\end{enumerate}

\noindent Thus 
$$l_n/d_n = \frac{C_Pn^{d+1}+0(n^{d})}{Vol(P)n^d+0(n^{d-1})}=\frac{nC_P}{Vol(P)} +0(1)$$
which proves (\ref{key}):
$$f_n(d)=\frac{(d+1)l_n}{ndd_n}= \frac{(d+1)}{d} \frac{l_n}{nd_n} \to  \frac{(d+1)}{d} \frac{C_P}{Vol(P)}.$$

\end{remark}

\subsection {Gram matrices and $P-$optimal measures.} \label{gramsec}  Let $E\subset \CC^d$ be closed and let $w$ be an admissible weight on $E$. We take $\mu$ a locally finite measure on $E$ and for each $n$ we define a weighted inner product on $Poly(nP)$:
\begin{equation}\label{Wip}
\langle f,g\rangle_{\mu,w}:=\int_E f(z)\overline{g(z)}w(z)^{2n}d\mu.
\end{equation}
We assume that $||f||_{L^2(w^nd\mu)}^2=\langle f,f\rangle_{\mu,w}<\infty$ for all $f\in Poly(nP)$ and that (\ref{Wip}) is non-degenerate in the sense that $||f||_{L^2(w^nd\mu)}=0$ implies $f\equiv 0$. Fixing a basis $\beta_n=\{p_1,p_2,\cdots, p_{d_n}\}$ of $Poly(nP)$ we form the Gram matrix
$$
G_n^{\mu,w}=G_n^{\mu,w}(\beta_n):=[\langle p_i,p_j\rangle_{\mu,w}]\in\CC^{d_n\times d_n}
$$
and the associated $n-$th Bergman function
\begin{equation}\label{WKn}
B_n^{\mu,w}(z):=\sum_{j=1}^{d_n}|q_j(z)|^2w(z)^{2n}
\end{equation}
where $Q_n=\{q_1,q_2,\cdots,q_{d_n}\}$ is an orthonormal basis for $Poly(nP)$ with respect
to the inner-product (\ref{Wip}). We make an observation which will be used in Lemma \ref{1stderiv} below. With this basis $\beta_n$, if we write
\begin{equation}\label{P}
P(z)=\left[\begin{array}{c}p_1(z)\cr p_2(z)\cr\cdot\cr\cdot\cr p_{d_n}(z)\end{array}\right]\in \CC^{ d_n}
\end{equation}
then 
\begin{equation}
\label{useful}
w(z)^{2n}P(z)^*\bigl(G_n^{\mu,w}(\beta_n)\bigr)^{-1}P(z)=B_n^{\mu,w}(z).
\end{equation}
To see this, $G:=G_n^{\mu,w}(\beta_n)$ and $G^{-1}$ are positive definite, Hermitian matrices; hence $G^{1/2}, \ G^{-1/2}:=(G^{-1})^{1/2}$ exist; writing $P:=P(z)$, we have
$$P^*G^{-1}P = P^*G^{-1/2}G^{-1/2}P=(G^{-1/2}P)^*G^{-1/2}P.$$
To verify that $w(z)^{2n}$ times the right-hand-side yields $B_n^{\mu,w}(z)$, note that since $G=\int_E PP^*w^{2n}d\mu$, the polynomials $\{\tilde p_1,\tilde p_2,\cdots, \tilde p_{d_n}\}$ defined by
\begin{equation}
G^{-1/2}P:=\left[\begin{array}{c}\tilde p_1(z)\cr 
\tilde p_2(z)\cr\cdot\cr\cdot\cr \tilde p_{d_n}(z)\end{array}\right]\in \CC^{ d_n}
\end{equation}
form an {\it orthonormal} basis for $Poly(nP)$ in $L^2(\mu)$: for
$$\int_E G^{-1/2}P \cdot (G^{-1/2}P)^*w^{2n}d\mu =  G^{-1/2}\bigl[\int_E PP^* w^{2n}d\mu \bigr] G^{1/2}$$
$$ = G^{-1/2} G G^{1/2}=I,$$
the $d_n\times d_n$ identity matrix. Thus 
$$B_n^{\mu,w}(z)= \sum_{j=1}^{d_n} |\tilde p_j(z)|^2w(z)^{2n} = w^{2n}(G^{-1/2}P)^*G^{-1/2}P.$$

Given $E$, and $w$ on $E$, for a function $u\in C(E)$, we consider the weight $w_t(z):=w(z)\exp(-tu(z)),$ $t\in\RR$. Apriori, $w_t$ need not be admissible. Let $\{\mu_n\}$ be a sequence of measures on $E$. Fixing a basis $\beta_n:=\{p_1,...,p_{d_n}\}$ of $Poly(nP)$, we set
\begin{equation}\label{fn}
f_n(t):=-{1\over 2l_n}\log\,{\rm det}(G_n^{\mu_n,w_t})
\end{equation}
where $G_n^{\mu_n,w_t}=G_n^{\mu_n,w_t}(\beta_n)$. We have the following result (Lemma 5.1 in \cite{[BB]} or Lemma 3.5 in \cite{bblw}) which will be used to prove Theorem \ref{keycn}.

\begin{lemma} \label{1stderiv} Suppose $w_t$ is admissible for $t$ in an interval containing $0$. For such $t$, we have
\[ f_n'(t)={n\over l_n}\int_E u(z)B_n^{\mu_n,w_t}(z)d\mu_n.\]
\end{lemma}
\begin{proof}
Recall that $G_n^{\mu_n,w_t}$ is a positive definite Hermitian matrix; hence we can define $\log (G_n^{\mu_n,w_t})$. Using $\log\,{\rm det}(G_n^{\mu_n,w_t})={\rm trace}\log(G_n^{\mu_n,w_t})$, we calculate
\begin{eqnarray*}
2l_nf_n'(t)&=&-{d\over dt}{\rm trace}\left(\log(G_n^{\mu_n,w_t})\right) \\
&=& -{\rm trace} \left({d\over dt}\log(G_n^{\mu_n,w_t})\right) \\
&=& -{\rm trace}\left( (G_n^{\mu_n,w_t})^{-1}{d\over dt} G_n^{\mu_n,w_t}\right) \\
\end{eqnarray*}
$$=2n\, {\rm trace}  \left( (G_n^{\mu_n,w_t})^{-1} \left[\int_E p_i(z)\overline{p_j(z)}u(z)w(z)^{2n}\exp(-2ntu(z))d\mu_n\right]\right).$$
We use $${\rm trace} (ABC) = {\rm trace} (CAB)= CAB$$ to write the previous line as
\begin{eqnarray*}
&=&2n\int_E P^*(z)(G_n^{\mu_n,w_t})^{-1}P(z) u(z)w(z)^{2n}\exp(-2ntu(z))d\mu_n\\
&=&2n\int_E u(z) P^*(z)(G_n^{\mu_n,w_t})^{-1}P(z) w_t(z)^{2n} d\mu_n \\
&=& 2n\int_E u(z) B_n^{\mu_n,w_t}(z) d\mu_n
\end{eqnarray*}
where the last equality follows from (\ref{useful}):
$$
w_t^{2n}P^*(G_n^{\mu_n,w_t})^{-1}P = B_n^{\mu_n,w_t}.
$$
\end{proof}

Similar, but more involved calculations, give the following (cf., Lemma 3.6 of \cite{bblw}).

\begin{lemma} \label{2ndderiv}
The functions $f_n(t)$ are concave, i.e., $f_n''(t)\le0.$
\end{lemma}

Now we restrict to $K\subset \CC^d$ compact and non-pluripolar. Fix $\mu$ a probability measure on $K$ and $w$ an admissible weight on $K$. If $\mu$ has the property that
\begin{equation}\label{wa} {\rm det}(G_n^{\mu',w})\le {\rm det}(G_n^{\mu,w})
\end{equation}
for all other probability measures $\mu'$ on $K$ then $\mu$ is said to be a {\it $P-$optimal measure of
degree $n$ for $K$ and $w$}. This property is independent of the basis used for $Poly(nP)$. An equivalent characterization is that 
\[\max_{z\in K} B_n^{\mu,w}(z)\le \max_{z\in K} B_n^{\mu',w}(z)\]
for all other probability measures $\mu'$ on $K$. Note that for {\it any} probability measure $\mu',$  $\displaystyle{\int_K B_n^{\mu',w}(z) d\mu'=d_n},$ so that
\[\max _{z\in K} B_n^{\mu',w}(z)\ge d_n.\]
For a $P-$optimal measure we have equality.

\begin{proposition}
\label{KW}
Let $w$ be an admissible weight on $K.$ A probability measure $\mu$
is a $P-$optimal measure of degree $n$ for $K$ and $w$ if and only if
$$ \max_{z\in K}B_n^{\mu,w}(z)=d_n.$$
\end{proposition} 

\noindent It follows that if $\mu$ is $P-$optimal for $K$ and $w$ then
\begin{equation} \label{wb} B_n^{\mu,w}(z)=d_n,\quad a.e.\,\, \mu.\end{equation}
We omit the proof; cf.,  \cite{[KW]} or Proposition 3.1 of \cite{bblw}.

\subsection {Ball volume ratios.} Given a (complex) $M-$dimensional vector space $V$, and two subsets $A,B$ in $V$, we write
$$[A:B]:=\log \frac{vol(A)}{vol(B)}$$
where ``vol'' denotes any (Haar) measure on $V$ (taking the ratio makes $[A:B]$ independent of this choice). In particular, if $V$ is equipped with two Hermitian inner products $h,h'$, and $B,B'$ are the corresponding unit balls, then a linear algebra exercise shows that 
\begin{equation}
\label{gramvolume}
[B:B'] = \log \det [h'(e_i,e_j)]_{i,j=1,...,M}
\end{equation}
where $e_1,...,e_M$ is an $h-$orthonormal basis for $V$. In other words, $[B:B'] $ is a {\it Gram determinant} with respect to the $h'$ inner product relative to the $h-$orthonormal basis. Indeed, $[B:B']$ is independent of the $h-$orthonormal basis chosen for $V$. 

{\it We will generally take $V=Poly(nP)$ and our subsets to be unit balls with respect to norms on $Poly(nP)$}; in this case we call (\ref{gramvolume}) a {\it ball volume ratio}. In particular, given $P$, let $\mu$ be a locally finite measure on a closed set $E\subset \CC^d$, and let $w$ be an admissible weight on $E$ such that (\ref{Wip}) is non-degenerate and $||f||_{L^2(w^nd\mu)}^2<\infty$ for all $f\in Poly(nP)$. We noted that for the unit torus $T^d$, the standard basis monomials $\beta_n=\{z^J, \ J\in nP\cap (\ZZ^+)^d\}$ form an orthonormal basis for $Poly(nP)$ with respect to the standard Haar measure $\mu_T$ on $T^d$. Letting
$$B_n=\{p_n\in Poly(nP): ||p_nw^n||_{L^2(\mu)}=||p_n||_{L^2(w^{2n}\mu)}\leq 1 \}$$
and
$$B_n'=\{p_n\in Poly(nP): ||p_n||_{L^2(\mu_T)}\leq 1 \}$$
be $L^{2}-$balls in $Poly(nP)$, we have
\begin{equation}\label{gramtorus} [B_n:B_n'] = \log \det G_n^{\mu,w}(\beta_n).\end{equation}
We will also use $L^{\infty}-$balls in $Poly(nP)$.

Taking $E=K$ compact and $\mu$ finite, replacing the standard
basis monomials $\{z^J, \ J\in nP\cap (\ZZ^+)^d\}$ by orthogonal polynomials $\{r_J(z)\}$ using the Gram-Schmidt process in $L^2(w^{2n}\mu)$, the Gram determinants ${\rm det}(G_n^{\mu,w})=\prod_J ||r_J||^2_{L^2(w^{2n}\mu)}$ are unchanged and we have 
$${\rm det}(G_n^{\mu,w})=\frac{1}{d_n!}Z_n:= \frac{1}{d_n!}Z_n(\mu,w)$$
where
$$Z_n:=\int_{K^{d_n}}|VDM(z_1,\cdots,z_{d_n})|^2w(z_1)^{2n}\cdots w(z_{d_n})^{2n} d\mu(z_1)\cdots d\mu(z_{d_n}).$$
It is easy to see that if $\mu$ is a Bernstein-Markov measure for the triple $(P,K,Q)$ where $w=e^{-Q}$, i.e., (\ref{wtdbm}) holds for $\mu$, then
\begin{equation}\label{estimates} Z_n \leq \delta^{w,n}(K)^{2l_n}\mu(K)^{d_n} \leq \mu(K)^{d_n} M_n^{2d_n} Z_n.\end{equation}

\begin{conjecture} \label{conj} Let $K\subset \CC^d$ be compact and let $w=e^{-Q}$ be an admissible weight on $K$. If $\mu$ is a Bernstein-Markov measure for the triple $(P,K,Q)$, then
\begin{equation}\label{gnlimit}
\lim_{n\to \infty} Z_n^{\frac{1}{2l_n}}=\lim_{n\to \infty} {\rm det}(G_n^{\mu,w})^{\frac{1}{2l_n}}=:\mathcal F_P(K,Q)
\end{equation}
exists. 
\end{conjecture}

We verify the conjecture in Remark \ref{rmk42}. It then follows from (\ref{estimates}) and (\ref{key}) that $\lim_{n\to \infty}\delta^{w,n}(K)$ exists and equals $\mathcal F_P(K,Q)$. This gives the existence of the limit in the definition of the $P-$transfinite diameter (\ref{tdlim}) and the weighted $P-$transfinite diameter  (\ref{deltaw}). We also have:

\begin{proposition}\label{tfd}
Let $K$ be compact and $w$ an admissible weight function. Assume (\ref{gnlimit}). For $n=1,2,...$, let $\mu_n$ be a $P-$optimal measure of order $n$ for $K$ and $w.$ Then
\[\lim_{n\to\infty}{\rm det}(G_n^{\mu_n,w})^{\frac{1}{2l_n}}=\mathcal F_P(K,Q).\]
\end{proposition}
\begin{proof}
We will use
$$\int_{K^{d_n}}|VDM(z_1,\cdots,z_{d_n})|^2w(z_1)^{2n}\cdots w(z_{d_n})^{2n} d\mu_n(z_1)\cdots d\mu_n(z_{d_n})$$
$$={d_n}!\,{\rm det}(G_n^{\mu_n,w}).$$ 
It follows, since $\mu_n$ is a probability measure, that
$$
{\rm det}(G_n^{\mu_n,w})\le {1\over {d_n}!}(\delta_n^w(K))^{2l_n}.
$$
Now if $f_1,f_2,\cdots,f_{d_n}\in K$ are weighted $P-$Fekete points of order $n$ for $K,$
i.e., points in $K$ for which 
\[|VDM(z_1,\cdots,z_{d_n})|w^n(z_1)\cdots w^n(z_{d_n})\]
is maximal, then the discrete measure 
\begin{equation}
\nu_n={1\over d_n}\sum_{k=1}^{d_n} \delta_{f_k}
\end{equation}
is a candidate for a $P-$optimal measure of order $n$; hence
\[ {\rm det}(G_n^{\nu_n,w})\le {\rm det}(G_n^{\mu_n,w}).\]
But
$$
{\rm det}(G_n^{\nu_n,w})={1\over d_n^{d_n}}|VDM(f_1,\cdots,f_{d_n})|^2w(f_1)^{2n}\cdots w(f_{d_n})^{2n} $$
$$={1\over d_n^{d_n}}(\delta^{w,n}(K))^{2l_n}
$$
so that
\[{1\over d_n^{d_n}}(\delta^{w,n}(K))^{2l_n} \le {\rm det}(G_n^{\mu_n,w}). \]
The result follows since (\ref{gnlimit}) implies $\lim_{n\to \infty}\delta^{w,n}(K)$ exists and equals $\mathcal F_P(K,Q)$.
\end{proof}

For future use we note that the ball volume ratios satisfy $[A:B]=-[B:A]$; the cocycle  condition:
$$[A:B]+[B:C]+[C:A]=0;$$
and they are ``monotone'' in the first slot: for any $B\subset Poly(nP)$, if $E\subset \CC^d$ is closed with admissible weights $Q_1\leq Q_2$ and 
$$\mathcal B^{\infty}(E,nQ_i):=\{p_n\in Poly(nP): ||p_ne^{-nQ_i}||_{E}\leq 1 \}, \ i=1,2$$
then
\begin{equation}\label{bvmon}[\mathcal B^{\infty}(E,nQ_1):B]\leq [\mathcal B^{\infty}(E,nQ_2):B]
\end{equation}
(with a similar statement for $L^2-$balls for $\mu$ a measure on $E$). Analogous properties will hold for the energy functional discussed next.

\section {Energy.}\label{enprop}

For $u,v \in L_{P,+}$, we define the {\it energy} 
\begin{equation}
\label{relendef}
\mathcal E (u,v):= \int_{\CC^d} (u-v)\sum_{j=0}^d (dd^cu)^j\wedge (dd^cv)^{d-j}.
\end{equation}
A reason for this definition will appear in Proposition \ref{diffprop}, and Theorem \ref{keycn} will relate asymptotics of certain ball volume ratios to the energy of appropriate $u,v$. For any functions $A,B\in L_{P,+}$ we have $A-B$ is uniformly bounded on $\CC^d$. We will need an integration by parts formula in this setting. Using results from Bedford-Taylor \cite{BT88}, one can show:  given $A,B,C,D\in L_{P,+}$, let $u_1,...,u_{d-1} \in L_{P,+}$.  Then
\begin{equation}
\label{ibyp1} \int_{\CC^d}(A-B)(dd^cC - dd^cD)\wedge dd^cu_1 \wedge \cdots \wedge dd^cu_{d-1} \end{equation} 
$$ =\int_{\CC^d}(C-D)(dd^cA - dd^cB)\wedge dd^cu_1 \wedge \cdots \wedge dd^cu_{d-1}$$
$$=-\int_{\CC^d}d(A-B)\wedge d^c(C - D)\wedge dd^cu_1 \wedge \cdots \wedge dd^cu_{d-1}.$$

The proof of the following fundamental differentiability property of the energy is exactly as that of Proposition 4.1 of \cite{[BB]}.

\begin{proposition} 
\label{diffprop}
Let $u,u',v\in L_{P,+}$. For $0\leq t\leq 1$, let 
$$f(t):= {\mathcal E} (u+t(u'-u),v).$$
Then $f'(t)$ exists for $0\leq t \leq 1$ and 
\begin{equation}
\label{difft} 
f'(t)= (d+1)\int_{\CC^d} (u'-u) (dd^c(u+t(u'-u)))^d.
\end{equation}
\end{proposition}

\begin{remark} Here we mean the appropriate one-sided derivatives at $t=0$ and $t=1$; e.g., 
\begin{equation}
\label{diff0} 
f'(0):=\lim_{t\to 0^+}\frac{f(t)-f(0)}{t}= (d+1)\int_{\CC^d} (u'-u) (dd^cu)^d.
\end{equation}
This last statement implies (\ref{difft}). For if $s$ is fixed,
$$g(t):=f(s+t)=   {\mathcal E} (u+(s+t)(u'-u),v)= {\mathcal E} (u+s(u'-u)+t(u'-u),v)$$
and applying (\ref{diff0}) to $g$ (so $u\to u+s(u'-u)$) we get 
$$g'(0)=f'(s)= (d+1)\int_{\CC^d} (u'-u) (dd^c(u+s(u'-u)))^d.$$
\end{remark}

We sometimes write (\ref{diff0}) in ``directional derivative'' notation as 
\begin{equation}
\label{diff02}
<\mathcal E' (u), u'-u>=(d+1)\int (u'-u)(dd^cu)^d.
\end{equation}
Note that the differentiation formula (\ref{difft}) is independent of $v$. This also follows from the {\it cocycle property}:

\begin{proposition} 
\label{cocycleprop}
Let $u,v,w\in L_{P,+}$. Then
$${\mathcal E}(u,v) +{\mathcal E}(v,w) + {\mathcal E}(w,u)=0.$$
\end{proposition}
\begin{proof} Let
$$f(t):={\mathcal E}(u+t(w-u),v)+{\mathcal E}(v,u)$$ 
and
$$g(t):={\mathcal E}(u+t(w-u),w)+{\mathcal E}(w,u).$$
Then $f(0)=g(0)=0$ by antisymmetry of ${\mathcal E}$. 
From (\ref{difft}),
$$f'(t)= (d+1)\int_{\CC^d} (w-u) (dd^c(u+t(w-u)))^d=g'(t)$$
for all $t$. Thus $f(1)=g(1)$; i.e.,
$${\mathcal E}(w,v) + {\mathcal E}(v,u)={\mathcal E}(w,w) + {\mathcal E}(w,u)={\mathcal E}(w,u).$$
\end{proof}

The independence of (\ref{difft}) on $v$ now follows: if $v,v'\in L_{P,+}$, then 
$${\mathcal E} (u+t(u'-u),v')+{\mathcal E} (v',v)+{\mathcal E} (v,u+t(u'-u))=0$$
so that the difference 
$${\mathcal E} (u+t(u'-u),v')-{\mathcal E} (u+t(u'-u),v)={\mathcal E} (v,v')$$
is independent of $t$. Thus we consider ${\mathcal E}$ as a functional on the first slot with the second fixed. As such, it is {\it increasing and concave}; the proof is exactly as for Proposition 4.4 of \cite{[BB]} and requires formula (\ref{ibyp1}).

\begin{proposition} Let $u,v,w\in L_{P,+}$. Then
$$u\geq v \ \hbox{implies} \ {\mathcal E}(u,w) \geq 
{\mathcal E}(v,w)$$
and for $0\leq t\leq 1$
$${\mathcal E}(tu+(1-t)v,w)\geq t{\mathcal E}(u,w) +(1-t){\mathcal E}(v,w);$$
i.e., $g(t):={\mathcal E}(tu+(1-t)v,w)$ satisfies $g''(t)\leq 0$. 
\end{proposition}

A consequence of concavity is the following. Let $u_1,u_2,v\in L_{P,+}$. Letting
$$g(s):= \mathcal E(u_1+s(u_2-u_1),v)$$
for $0\leq s\leq 1$, we have concavity of $g$ so that $g(s)\leq g(0)+g'(0)s$. In particular, at $s=1$, we have $
g(1)\leq g(0)+g'(0)$; i.e.,
\begin{equation}\label{gconcave}\mathcal E(u_2,v)\leq \mathcal E(u_1,v)+(d+1)\int_{\CC^d}(u_2-u_1)(dd^cu_1)^d.
\end{equation}

For future use, we record the following. 

\begin{lemma}
\label{BT63} Let $\{w_j\},\{v_j\}\subset L_{P,+}$ with $w_j\uparrow w\in L_{P,+}$ and $v_j \uparrow v \in L_{P,+}$. Then
$$\mathcal E(w_j,v) \to \mathcal E(w,v) \ \hbox{and} \ \mathcal E(w_j,v_j) \to \mathcal E(w,v).$$
\end{lemma}
\begin{proof} From Proposition \ref{cocycleprop}, it suffices to prove the first statement. This follows directly from the {\it proof} of Lemma 6.3 of \cite{BT88}: {\it given $$w,\{v_j\},v, \{u_{1,j}\},u_1,...,\{u_{d,j}\},u_d \ \hbox{in} \ L_{P,+}$$
with $v_j\uparrow v, \ u_{1,j}\uparrow u_1, ..., u_{d,j}\uparrow u_d$, 
$$\lim_{j\to \infty} \int_{\CC^d} (w-v_j)dd^cu_{1,j}\wedge \cdots \wedge dd^cu_{d,j}=\int_{\CC^d} (w-v)dd^cu_{1}\wedge \cdots \wedge dd^cu_{d}.$$}
\end{proof}

\noindent We remark that if $w_j\downarrow w\in L_{P,+}$ and $v_j \downarrow v \in L_{P,+}$ then we still have
\begin{equation}
\label{downlemma}
\mathcal E(w_j,v) \to \mathcal E(w,v) \ \hbox{and} \ \mathcal E(w_j,v_j) \to \mathcal E(w,v).
\end{equation}
The first statement is standard and the second follows from the first by Proposition \ref{cocycleprop}.

\section {Differentiability of $\mathcal E \circ \Pi$.} \label{diffep}

We turn to the main differentiability result. Our exposition mimics Lemmas 4.10 and 4.11 of  \cite{[BB]}; since this is the key ingredient in proving Theorem \ref{keycn} we include all details. Generally we will fix a function $v\in L_{P,+}$ which will be in the second slot of all energy terms and we simply write, for any $\tilde v\in L_{P,+}$,
$${\mathcal E}(\tilde v):={\mathcal E}(\tilde v,v).$$
If we need to emphasize a specific $v$, we revert to the notation on the right-hand-side of this equation. Recall for $E\subset \CC^d$ closed and an admissible weight $a$ on $E$, we write $\Pi(a)$ (sometimes $\Pi_E(a)$) to denote the regularized weighted $P-$extremal function $V_{P,E,a}^*$.

We state two versions of differentiability of $\mathcal E \circ \Pi$. One version, Proposition \ref{mainprop}, is for a second admissible weight $b$ on $E$ where we consider the perturbed weight $a+t(b-a)$ and the associated weighted $P-$extremal function $\Pi(a+t(b-a))$ and we show the differentiability of 
$$F(t):=\mathcal E( \Pi(a+t(b-a))).$$
Taking $v=\Pi(a)$, as we will in Propositions \ref{mainprop}, \ref{diffpropcor} and Lemma \ref{keylemma2}, 
\begin{equation}\label{zero}F(0)=\mathcal E( \Pi(a))=\mathcal E( \Pi(a),\Pi(a))=0.\end{equation}
If $E$ is unbounded, we will need to make an additional assumption on $u:=b-a$ so that (\ref{loclip}) holds; also, in this case, we restrict to $0\leq t\leq 1$ so that $a+t(b-a)=tb+(1-t)a$, being a convex combination of $a, b$, is admissible on $E$. The second version of differentiability for $\mathcal E \circ \Pi$, Proposition \ref{diffpropcor}, is for a compact set $K$ and an arbitrary real $t$. We take a function $u\in C(K)$, consider the perturbed weight $a+tu$, and show the differentiability of 
$$F(t):=\mathcal E( \Pi(a+tu)).$$
Apriori, since $t\in \RR$, we must assume $u$ is continuous so that $a+tu$ is an admissible (lowersemicontinuous) weight. The following results utilize Lemma \ref{mamass} and Corollary \ref{mamassbis}; hence we assume $C^2-$regularity of $a,b$ and/or $u$. 
  
\begin{proposition} 
\label{mainprop}
Let $v\in L_{P,+}$. For admissible weights $a,b\in C^2(E)$ on a closed set $E\subset \CC^d$, let $u:=b-a$ and let 
$$F(t):={\mathcal E}(\Pi(a+tu),v))$$
for $t\in \RR$. If $E$ is unbounded, assume (\ref{unbhyp}) holds and $0\leq t\leq 1$. Then
\begin{equation}
\label{Pdifft}
F'(t)= (d+1)\int_{\CC^d} u (dd^c \Pi(a+tu))^d.
\end{equation}

\end{proposition}

\begin{proposition}
\label{diffpropcor}
Let $v\in L_{P,+}$. For an admissible weight $a$ on a compact set $K\subset \CC^d$ and $u\in C^2(K)$, let 
$$F(t):={\mathcal E}(\Pi(a+tu),v)$$
for $t\in \RR$. Then
\begin{equation}
\label{Pdifftcor}
F'(t)= (d+1)\int_{\CC^d} u (dd^c \Pi(a+tu))^d.
\end{equation}
\end{proposition}

We prove Propositions \ref{mainprop} and \ref{diffpropcor} simultaneously.

\begin{proof} We may take $v=\Pi(a)$. As in the proof of Proposition \ref{diffprop} we prove only the one-sided limit as $t\to 0^+$:  
\begin{equation}
\label{Pdiff0} 
F'(0):=\lim_{t\to 0^+}\frac{F(t)-F(0)}{t}= (d+1)\int_{\CC^d} u (dd^c \Pi(a))^d.
\end{equation}
This implies (\ref{Pdifft}). For if $s$ is fixed,
$$G(t):=F(s+t)=   {\mathcal E}(\Pi(a+(s+t)u),v))$$
$$= {\mathcal E}(\Pi(a+su+tu),v))$$
and applying (\ref{Pdiff0}) to $G$ (so $a\to a+su$) we get 
$$G'(0)=F'(s)= (d+1)\int_{\CC^d} u (dd^c \Pi(a+su))^d.$$

Note that $F(0)=0$ (see (\ref{zero})) and to verify (\ref{Pdiff0}) it suffices to prove 
\begin{equation}
\label{smooth5.7at0}
\mathcal E(\Pi(a+tu),\Pi(a))=(d+1)t\int_{\CC^d} u(dd^c\Pi(a))^d +o(t).
\end{equation}
We need two ingredients for (\ref{smooth5.7at0}): 
\begin{equation}
\label{item1}
\mathcal E(\Pi(a+tu),\Pi(a))=(d+1)\int_{\CC^d} [\Pi(a+tu)-\Pi(a)](dd^c\Pi(a))^d +o(t)
\end{equation}
and
\begin{equation}
\label{item2}
\lim_{t\to 0} \int_{D(0)\setminus D(t)} (dd^c\Pi(a))^d  =0
\end{equation}
where
$$D(t):=\{z \in \CC^d: \ \Pi(a+tu)(z) = (a+tu)(z)\}.$$
We have proved (\ref{item2}) in Lemma \ref{mamass}.

We state and prove (\ref{item1}) in a separate lemma. Given (\ref{item1}) and (\ref{item2}), and observing from (\ref{suppw}) that 
\begin{equation}
\label{support}
\hbox{supp}(dd^c\Pi(a))^d  \subset D(0),
\end{equation}
(\ref{smooth5.7at0}) follows as in \cite{[BB]}, p. 28:
$$\mathcal E(\Pi(a+tu),\Pi(a))=(d+1)\int_{\CC^d} [\Pi(a+tu)-\Pi(a)](dd^c\Pi(a))^d +o(t)$$
$$=(d+1)\int_{D(0)\setminus D(t)} [\Pi(a+tu)-\Pi(a)](dd^c\Pi(a))^d $$
$$+(d+1)\int_{D(0)\cap D(t)} [\Pi(a+tu)-\Pi(a)](dd^c\Pi(a))^d +o(t)$$
$$=(d+1)\int_{D(0)\setminus D(t)} [\Pi(a+tu)-\Pi(a)](dd^c\Pi(a))^d$$
$$ +(d+1)t\int_{D(0)\cap D(t)} u(dd^c\Pi(a))^d +o(t)$$
$$=(d+1)\int_{D(0)\setminus D(t)} [\Pi(a+tu)-\Pi(a)-tu](dd^c\Pi(a))^d$$
$$ +(d+1)t\int_{D(0)} u(dd^c\Pi(a))^d +o(t)$$
since $\Pi(a+tu)-\Pi(a)= tu$ on $D(0)\cap D(t)$. Now (\ref{loclip}) or (\ref{loclip2}) implies 
$$|\Pi(a+tu)-\Pi(a)-tu| =0(t)$$
on the {\it bounded} set $D(0)\setminus D(t)$
(recall if $E$ is unbounded we assume (\ref{unbhyp}) holds in the setting of Proposition \ref{mainprop}) and this fact, combined with (\ref{item2}) and (\ref{support}), finishes the proof. 
\end{proof}

In (\ref{item1}), since $(dd^c\Pi(a))^d$ is supported in $D(0)$, 
$$\int_{\CC^d} [\Pi(a+tu)-\Pi(a)](dd^c\Pi(a))^d= \int_{D(0)} [\Pi(a+tu)-\Pi(a)](dd^c\Pi(a))^d;$$
and, on $D(t)\cap D(0)$, we have $\Pi(a+tu)-\Pi(a)=tu$. The  content of (\ref{item2}) is that the contribution to this integral on $D(0)\setminus D(t)$ is negligible. The content of (\ref{item1}), Lemma \ref{keylemma2} below, is that the contribution of each of the $d+1$ terms in the energy $\mathcal E(\Pi(a+tu),\Pi(a))$ is the same, up to $o(t)$, as that involving the term $(dd^c\Pi(a))^d$. Again we write
$$F(t):=\mathcal E(\Pi(a+tu)) =\mathcal E(\Pi(a+tu),\Pi(a))$$
$$=\int [\Pi(a+tu)-\Pi(a)][(dd^c \Pi(a+tu))^d +... + (dd^c \Pi(a))^d].$$
Another interpretation of (\ref{item1}) is that to prove the differentiability of $\mathcal E \circ \Pi$, we can replace $\mathcal E$ by its ``linearization'' at $\Pi(a)$. As in previous arguments, we only give the proof at $t=0$ and for the one-sided limit in (\ref{Pdifftcor}) as $t\to 0^+$. The next result does not require smoothness of $u$.

\begin{lemma} 
\label{keylemma2}
For an admissible weight $a$ on $E$ and $u\in C(E)$, let
$$F(t)=\mathcal E(\Pi(a+tu)) $$
$$=\int [\Pi(a+tu)-\Pi(a)][(dd^c \Pi(a+tu))^d +... + (dd^c \Pi(a))^d]$$
and
$$G(t):= (d+1)\int [\Pi(a+tu)-\Pi(a)](dd^c \Pi(a))^d.$$ 
Then
$$\lim_{t\to 0^+} \frac{F(t)-F(0)}{t}=\lim_{t\to 0^+} \frac{G(t)-G(0)}{t}.$$
\end{lemma}

\begin{proof} Note that $F(0)=\mathcal E(\Pi(a))=0$ and $G(0)=0$. By concavity of $\Pi$ (recall (\ref{projcon})) and linearity of 
$f\to \int f (dd^c \Pi(a))^d$, the function $G(t)$ is concave so that 
\begin{equation}\label{whatisa} A:=\lim_{t\to 0^+} \frac{G(t)-G(0)}{t}\end{equation}
exists. By concavity of $\mathcal E$, we have (recall (\ref{diff02}))
$$\mathcal E(\Pi(a+tu))\leq \mathcal E(\Pi(a))+<\mathcal E'(\Pi(a)), \Pi(a+tu)-\Pi(a)>;$$
i.e., from (\ref{gconcave}) with $u_1=\Pi(a), \ u_2 = \Pi(a+tu)$ and $v= \Pi(a)$, 
$$\mathcal E(\Pi(a+tu))\leq \mathcal E(\Pi(a))+ (d+1)\int [\Pi(a+tu)-\Pi(a)](dd^c \Pi(a))^d].$$
Thus
$$\limsup_{t\to 0^+} \frac{F(t)-F(0)}{t}\leq A.$$
We prove 
$$\liminf_{t\to 0^+} \frac{F(t)-F(0)}{t}\geq A.$$
Since $A:=\lim_{t\to 0^+} \frac{G(t)-G(0)}{t}$ exists, given $\epsilon >0$ we can choose $\delta >0$ sufficiently small so that 
$$\frac{G(\delta)-G(0)}{\delta}=\frac{d+1}{\delta}\int [\Pi(a+\delta u)-\Pi(a)](dd^c \Pi(a))^d\geq A -\epsilon;$$
i.e.,
$$(d+1)\int [\Pi(a+\delta u)-\Pi(a)](dd^c \Pi(a))^d\geq \delta (A -\epsilon).$$
From Proposition \ref{diffprop}, for $t>0$ sufficiently small we have
$$\frac{\mathcal E(\Pi(a) +t[\Pi(a+\delta u)-\Pi(a)])-\mathcal E(\Pi(a))}{t}$$
$$\geq (d+1)\int [\Pi(a+\delta u)-\Pi(a)](dd^c\Pi(a))^d-\delta \epsilon;$$
i.e.,
$$\mathcal E((1-t)\Pi(a) +t\Pi(a+\delta u))=\mathcal E(\Pi(a) +t[\Pi(a+\delta u)-\Pi(a)])$$
$$\geq \mathcal E(\Pi(a))+t(d+1) \int [\Pi(a+\delta u)-\Pi(a)](dd^c\Pi(a))^d-t\delta \epsilon.$$
Combining these last two inequalities, we have
$$\mathcal E((1-t)\Pi(a) +t\Pi(a+\delta u))\geq \mathcal E(\Pi(a))+t\delta A -2t\delta \epsilon.$$
By concavity of $\Pi$,
$$\Pi(a+t\delta u)=\Pi((1-t)a +t(a+\delta u ))\geq (1-t)\Pi(a) +t\Pi(a +\delta u)$$
so that, by monotonicity of $\mathcal E$,
$$\mathcal E (\Pi(a+t\delta u))\geq \mathcal E((1-t)\Pi(a) +t\Pi(a +\delta u))\geq \mathcal E(\Pi(a))+t\delta A -2t\delta \epsilon$$
for $t>0$ sufficiently small. Thus,
$$\liminf_{t\to 0^+} \frac{F(t)-F(0)}{t}\geq A - 2\epsilon$$
for all $\epsilon >0$, yielding the result. \end{proof}

We now finish the proof of Proposition \ref{mainprop} and Proposition  \ref{diffpropcor} by finding $A$ in (\ref{whatisa}). The proof that $A= \int u(dd^c\Pi(a))^d$ was essentially given in the verification of (\ref{smooth5.7at0}) assuming (\ref{item1}) and (\ref{item2}); for the reader's convenience, we give the details. We write $S_a:=$supp$(dd^c\Pi(a))^d$.
For each $t$, $D(t)=\{z\in \CC^d: \Pi(a+tu)(z) = a(z) +tu(z)\}$ is a bounded set. From Proposition \ref{turgay3}, $\Pi(a)=a$ a.e.-$(dd^c\Pi(a))^d$ on $S_a\subset D(0)$; thus
$$\int [\Pi(a+tu)-\Pi(a)] (dd^c\Pi(a))^d = \int_{S_a} [\Pi(a+tu)-\Pi(a)] (dd^c\Pi(a))^d$$
$$=\int_{D(t)\cap S_a} [\Pi(a+tu)-\Pi(a)] (dd^c\Pi(a))^d$$
$$ + \int_{S_a\setminus D(t)} [\Pi(a+tu)-\Pi(a)](dd^c\Pi(a))^d$$
$$=\int_{D(t)\cap S_a} [a+tu-a] (dd^c\Pi(a))^d + \int_{S_a\setminus D(t)} [\Pi(a+tu)-\Pi(a)] (dd^c\Pi(a))^d$$
$$=\int_{D(t)\cap S_a} tu (dd^c\Pi(a))^d + \int_{S_a\setminus D(t)} [\Pi(a+tu)-\Pi(a)] (dd^c\Pi(a))^d$$
$$=\int_{S_a}tu (dd^c\Pi(a))^d + \int_{S_a\setminus D(t)} [\Pi(a+tu)-\Pi(a)-tu] (dd^c\Pi(a))^d.$$
Now we use the observation (\ref{loclip}) (or (\ref{loclip2})) to see that 
$$|\Pi(a+tu) - \Pi(a)-tu|=0(t)$$
on the bounded set $S_a\setminus D(t)$; the conclusion follows from Lemma \ref{mamass}.

We record an integrated version of Proposition \ref{mainprop} and Proposition \ref{diffpropcor} which we will use. 

\begin{proposition} 
\label{smoothample}
For admissible weights $a,b\in C^2(E)$ on an unbounded closed set $E$ satisfying (\ref{unbhyp}),
\begin{equation}
\label{smooth5.7}
\mathcal E(\Pi(b),\Pi(a))=(d+1)\int_{t=0}^1 dt \int_{\CC^d} (b-a)(dd^c\Pi(a+t(b-a)))^d;
\end{equation}
and for a compact set $K$ with admissible weight $a$ and $u\in C^2(K)$, 
\begin{equation}
\label{smooth5.8}
\mathcal E(\Pi(a+u),\Pi(a))=(d+1)\int_{t=0}^1 dt \int_{\CC^d} u(dd^c\Pi(a+tu))^d.
\end{equation}
\end{proposition}

\begin{proof} We prove (\ref{smooth5.7}) as (\ref{smooth5.8}) is similar. We begin with Proposition \ref{mainprop} using $v=\Pi(a)$ so that 
$F(t)={\mathcal E}(\Pi(a+t(b-a)),\Pi(a))$ and (\ref{Pdifft}) becomes
$$F'(t)=(d+1)\int_{\CC^d} (b-a)(dd^c\Pi(a+t(b-a)))^d.$$
Integrating this expression from $t=0$ to $t=1$ gives (\ref{smooth5.7}) since $F(1)-F(0)=\mathcal E(\Pi(b),\Pi(a))$.
\end{proof}

\section {The Main Theorem.}\label{mainth}

In this section, we state and prove the main result which relates asymptotics of certain ball-volume ratios with energies associated with $P-$extremal functions. For $E\subset \CC^d$ closed, following notation in \cite{[BB]}, we let $\phi$ be an admissible weight on $E$. Let
$$\mathcal B^{\infty}(E,n\phi):=\{p_n\in Poly(nP): |p_n(z)^2e^{-2n\phi(z)}|\leq 1 \ \hbox{on}  \ E\}$$
be an $L^{\infty}-$ball and, if $\mu$ is a measure on $E$, let
$$\mathcal B^{2}(E,\mu,n\phi):=\{p_n\in Poly(nP): \int_E |p_n|^2e^{-2n\phi}d\mu\leq 1 \}$$
be an $L^{2}-$ball in $Poly(nP)$. The key result is the following.

\begin{theorem}
\label{keycn} Given $\phi,\phi'$ admissible weights on $E,E'$,
$$\lim_{n\to \infty}\frac{-(d+1)n_d}{2nd_n} [\mathcal B^{\infty}(E,n\phi):\mathcal B^{\infty}(E',n\phi')]= \mathcal E(V_{P,E,\phi}^*, V_{P,E',\phi'}^*).$$
If $\mu,\mu'$ are measures on $E,E'$ where $\mu$ is a Bernstein-Markov measure for $(P,E,\phi)$ and $\mu'$ is a Bernstein-Markov measure for $(P,E',\phi')$, then
$$\lim_{n\to \infty} \frac{-(d+1)n_d}{2nd_n} [\mathcal B^{2}(E,\mu,n\phi):\mathcal B^{2}(E',\mu',n\phi')]= \mathcal E(V_{P,E,\phi}^*,V_{P,E',\phi'}^*).$$
\end{theorem}

\begin{remark} \label{rmk42} Taking $E'=T$ and $\phi' =0$, from (\ref{torusextr}) we have $V_{P,E',\phi'}^*=H_P$. Now taking $\mu'=\mu_T$ and taking $(K,\mu,Q)$ for the triple $(E,\mu,\phi)$ where $K$ is compact and $\mu$ is a Bernstein-Markov measure for $(P,K,Q)$, we verify Conjecture \ref{conj}. We use (\ref{gramtorus}) and (\ref{key}) to obtain (\ref{gnlimit}), the existence of the limit
\begin{equation}\label{wbmfinal} \lim_{n\to\infty}{\frac{1}{2l_n}} \log {\rm det}(G_n^{\mu,w})=\frac{-1}{n_dd\mathcal A}\mathcal E(V_{P,K,Q}^*,H_P)=\log {\mathcal F_P(K,Q)}.\end{equation}
Thus we obtain the asymptotics of weighted Gram determinants associated to $(K,\mu,Q)$ as well as the other results mentioned in Section \ref{sec:back}: the existence of the limit of the scaled maximal weighted Vandermondes
$$ \delta^w(K):=\lim_{n\to \infty}\delta^{w}_n(K)=\mathcal F_P(K,Q)$$
in (\ref{deltaw}) and Proposition \ref{tfd} on $P-$optimal measures.

\end{remark}

The first step of the proof is a version of Bergman asymptotics in a special case.

\subsection {Weighted Bergman asymptotics in $\CC^d$.}\label{sec:bergcd}

We state a result on Bergman asymptotics in \cite{[B1]}. The setting is this: $\phi\in C^{1,1}(\CC^d)$ with 
\begin{equation}
\label{stradm}
\phi (z) \geq (1+\epsilon)H_P(z) \ \hbox{for}  \ |z|>>1 \ \hbox{for some}  \ \epsilon >0.
\end{equation}
We will call a global admissible weight $\phi$ satisfying (\ref{stradm}) {\it strongly admissible}. For $p_n\in Poly(nP)$, we write 
$$||p_n||_{n\phi}^2:=||p_n||_{\omega_d,n\phi}^2=\int_{\CC^d} |p_n(z)|^2 e^{-2n\phi(z)}\omega_d(z)$$
where $\omega_d$ is Lebesgue measure on $\CC^d$. Using (\ref{sigmainkp2}), under the growth assumption on $\phi$, if $n>\frac{d}{\epsilon kA}$ where $P \subset A\Sigma$ then for each polynomial $p_n\in Poly(nP)$, $||p_n||_{n\phi}<+\infty$. 

Given an orthonormal basis $\{q_1,...,q_{d_n}\}$ of $Poly(nP)$, in this section we use the notation
$$B_{n,\phi}(z):=[\sum_{j=1}^{d_n} |q_j(z)|^2]e^{-2n\phi(z)}$$
for the $n$-th Bergman function; and we recall that
$$B_{n,\phi}(z)=\sup_{p_n\in Poly(nP)\setminus \{0\}} |p_n(z)|^2e^{-2n\phi(z)}/||p_n||_{n\phi}^2.$$
Finally, let 
$$S:=\{z\in \CC^d: dd^c\phi(z) \ \hbox{exists and} \ dd^c\phi(z)>0\}$$
and if $u$ is a $C^{1,1}$ function such that $(dd^cu)^d$ is absolutely continuous with respect to Lebesgue measure, we write
$$\det (dd^c u)\omega_d :=(dd^c u)^d.$$

\begin{theorem}
\label{bermancn} Given $\phi\in C^{1,1}(\CC^d)$ satisfying (\ref{stradm}), we have the following: $V_{P,\CC^d,\phi}\in C^{1,1}(\CC^d)$;  $(dd^cV_{P,\CC^d,\phi})^d$ has compact support and is absolutely continuous with respect to Lebesgue measure; 
$$(dd^cV_{P,\CC^d,\phi})^d=\det (dd^cV_{P,\CC^d,\phi})\omega_d$$
as $(d,d)-$forms with $L^{\infty}_{loc}(\CC^d)$ coefficients; and a.e. on the set $D:=\{V_{P,\CC^d,\phi} =\phi\}$ we have $\det (dd^c\phi)=\det (dd^cV_{P,\CC^d,\phi})$. Moreover, 
$$ \frac{n_d}{d_n}B_{n,\phi}\to \chi_{D\cap S} \det (dd^c\phi) \ \hbox{in} \  L^1(\CC^d)$$ and the measures
$$\frac{n_d}{d_n}B_{n,\phi}\omega_d \to (dd^cV_{P,\CC^d,\phi})^d \ \hbox{weakly}.$$
\end{theorem}

\noindent Recall the (strong) admissibility of $\phi$ implies, by Proposition \ref{turgay3}, that $(dd^cV_{P,\CC^d,\phi})^d$ has compact support.

\begin{remark}\label{import} From \cite{bloom}, $(D,\omega_d|_D,\phi|_D)$ satisfies a weighted Bernstein-Markov property for $\mathcal P_n$ or $A\mathcal P_n$; from Remark \ref{ap}, $\omega_d|_D$ is a Bernstein-Markov measure for the triple $(P,D,\phi)$. Using Proposition \ref{turgay4},  
$$\sup_{\CC^d}|p_n e^{-n\phi}|=\sup_{D}|p_n e^{-n\phi}|$$
for $p_n\in Poly(nP)$. Hence, from (\ref{wtdbm}), 
$$\sup_{\CC^d}|p_n e^{-n\phi}| \leq M_n[\int_{D}|p_n|^2e^{-2n\phi}\omega_d]^{1/2}\leq M_n [\int_{\CC^d}|p_n|^2e^{-2n\phi}\omega_d]^{1/2}$$
where $M_n^{1/n}\to 1$. This last integral is finite by (\ref{sigmainkp2}). 
\end{remark}

\subsection {Proof of the Main Theorem.} \label{pomt}
We consider several cases.
\medskip

\noindent{\sl Case 1: $E=E'=\CC^d$ and $\phi, \phi'\in C^2(\CC^d)$ strongly admissible with $\phi'=\phi$ outside a ball ${\mathcal B}_R$ for some $R$; $d\mu=d\mu'=\omega_d$}:
\medskip

We begin in the $L^2-$Case 1. Note that (\ref{unbhyp}) holds for then all of the weights $\phi+t(\phi'-\phi)$ are strongly admissible with a uniform $\epsilon$ (recall (\ref{stradm})). Let $u:=\phi'-\phi$; then $u$ is continuous with compact support. For $0\leq t\leq 1$ let 
$$\phi_t:=\phi +tu= \phi +t(\phi'-\phi)=(1-t)\phi + t\phi'$$
so that $\phi_0=\phi$ and $\phi_1=\phi'$; equivalently, $w_t(z):=w(z)\exp(-tu(z))$ (note $w_0 =w=e^{-\phi}$ and $w_1=w'=e^{-\phi'}$). Then from Theorem \ref{bermancn}, for each $t$, 
$$\frac{n_d}{d_n}B_{n,\phi+tu}\cdot  \omega_d \to (dd^c  \Pi(\phi+tu))^d \ \hbox{weakly}.$$

Now set $$f_n(t):= -{1\over 2l_n}\log\,{\rm det}(G_n^{\mu,w_t}(\beta_n))$$ where $\mu=\mu_n:= \omega_d$ for all $n$ and the basis $\beta_n:=\{p_1,...,p_{d_n}\}$ of $Poly(nP)$ is chosen to be an orthonormal basis with respect to the weighted $L^2-$norm $p\to ||w^np||_{L^2(\mu)}$. Then $G_n^{\mu,w}(\beta_n)$ is the $d_n\times d_n$ identity matrix so that we have $f_n(0)=0$; and, using Lemma \ref{1stderiv} and the fact that $u$ has compact support (thus all weights $w_t$ are admissible), 
$$\lim_{n\to \infty}\frac{l_n}{nd_n} f_n'(t)=\lim_{n\to \infty}\frac{1}{d_n} \int uB_{n,\phi+tu}\omega_d =\frac{1}{n_d} \int u(dd^c  \Pi(\phi+tu))^d .$$
We now integrate $\frac{l_n}{nd_n} f_n'(t)$ from $t=0$ to $t=1$:
$$\frac{l_n}{nd_n} [f_n(1) -f_n(0)]=\frac{l_n}{nd_n} [f_n(1)]= \frac{-1}{2nd_n}\log\,{\rm det}(G_n^{\mu,w'}(\beta_n))$$
$$= \frac{-1}{2nd_n}[\mathcal B^{2}(\CC^d,\mu,n\phi):\mathcal B^{2}(\CC^d,\mu,n\phi')] \ \hbox{(from (\ref{gramvolume}))}$$
$$=\frac{1}{d_n}\int_{t=0}^1 dt \int  B_{n,\phi+tu}(\phi-\phi')\omega_d \ \hbox{(from Lemma \ref{1stderiv})}$$
$$\to \frac{1}{n_d}\int_{t=0}^1 dt \int (\phi-\phi')(dd^c  \Pi(\phi+tu))^d.$$
But by (\ref{smooth5.7}), since (\ref{unbhyp}) holds, 
$$(d+1)\int_{t=0}^1 dt \int (\phi-\phi')(dd^c  \Pi(\phi+tu))^d=\mathcal E(\Pi(\phi'),\Pi(\phi))$$
which proves Theorem \ref{keycn} in $L^2-$Case 1. By Remark \ref{import} this also proves the $L^{\infty}-$Case 1.  
\medskip

\noindent{\sl Case 2: $E=E'=\CC^d$ and $\phi, \phi'\in C^2(\CC^d)$ strongly admissible; $d\mu=d\mu'=\omega_d$}: 
\medskip

We first do the $L^{\infty}-$Case 2. Remark \ref{compsupp} and Proposition \ref{turgay2} imply that $$\Pi(\phi)=\Pi_{S_w}(\phi|_{S_w})$$
where $S_w=\hbox{supp}(dd^c\Pi(\phi))^d$ is compact; moreover, 
for $p_n\in Poly(nP)$, from Proposition \ref{turgay4}, $||p_ne^{-n\phi}||_{S_w}=||p_ne^{-n\phi}||_{\CC^d}$ so that
$${\mathcal B}^{\infty}(S_w,n\phi|_{S_w})={\mathcal B}^{\infty}(\CC^d,n\phi).$$
Thus modifying $\phi, \phi'$ outside a large ball in such a way to make them equal outside a perhaps larger ball, we neither change the $L^{\infty}-$ball volume ratios nor the $P-$extremal functions $\Pi(\phi), \Pi(\phi')$. Hence the $L^{\infty}-$Case 2 follows from the $L^{\infty}-$Case 1. By Remark \ref{import}  this also proves the $L^2-$Case 2.
\medskip

\noindent{\sl Case 3 (general): $E,E'\subset \CC^d$ closed with admissible weights $\phi,\phi'$; $\mu, \mu'$ Bernstein-Markov measures for $(P,E,\phi),(P,E',\phi')$}:
\medskip

We consider the $L^{\infty}-$Case 3 only; the $L^2-$Case 3 follows from the definition of Bernstein-Markov measure for $(P,E,\phi),(P,E',\phi')$. We claim that by the cocycle property for the ball volume ratios $[A:B]$ and energies $\mathcal E(u_1,u_2)$, we may assume that one of the sets is $\CC^d$  
with a strongly admissible $C^2(\CC^d)$ weight $\hat \phi$. For, using the notation $\Pi_E(\phi):=V_{P,E,\phi}^*$, we have
$$ \mathcal E(\Pi_E(\phi),\Pi_{E'}(\phi')) =- \mathcal E(\Pi_{E'}(\phi'),\Pi_{\CC^d}(\hat \phi)) + \mathcal E(\Pi_E(\phi),\Pi_{\CC^d}(\hat \phi)).$$
Both terms on the right have the second term being $\Pi_{\CC^d}(\hat \phi)$. Similarly, with respect to the ball volume ratios, for each $n$ we have
$$[{\mathcal B}^{\infty}(E,n\phi):{\mathcal B}^{\infty}(E',n\phi')] $$
$$=-[{\mathcal B}^{\infty}(E',n\phi'): {\mathcal B}^{\infty}(\CC^d,n\hat \phi)]+ [{\mathcal B}^{\infty}(E,n\phi):{\mathcal B}^{\infty}(\CC^d,n\hat \phi)].$$

Now to deduce the case where one of the sets is $\CC^d$  
with a strongly admissible $C^2(\CC^d)$ weight $\hat \phi$ and the other is a general closed set $E$ with admissible weight $\phi$ from Case 2 where both sets are $\CC^d$ 
with strongly admissible $C^2(\CC^d)$ weights $\hat \phi, \psi$, we first observe that we may assume $E$ is compact (i.e., bounded). For
recall again from Proposition \ref{turgay2} that if $w=e^{-\phi}$, $\Pi_E(\phi)=\Pi_{S_w}(\phi|_{S_w})$ where $S_w=\hbox{supp}(dd^c\Pi_E(\phi))^d$ is compact; and for $p_n\in Poly(nP)$, $||p_ne^{-n\phi}||_{S_w}=||p_ne^{-n\phi}||_E$ so that 
$${\mathcal B}^{\infty}(S_w,n\phi|_{S_w})={\mathcal B}^{\infty}(E,n\phi).$$
Thus we assume $E$ is compact; since $V_{P,E,\phi}^*\in L_{P,+}$, we can also assume $\phi$ is bounded above on $E$. We take a large sublevel set $B_R:=\{z\in \CC^d: H_P(z) < \log R\}$ containing $E$ and extend $\phi$ from $E$ to $\hat \psi$ on $\CC^d$:
$$\hat \psi := \phi \ \hbox{on} \ E; \ \hat \psi =2\log R  \ \hbox{on} \ B_R \setminus E; \ \hat \psi =2kH_P(z) \ \hbox{on} \ \CC^d \setminus B_R.$$
We have $\hat \psi$ is lowersemicontinuous and by taking $R$ sufficiently big $\Pi_{\CC^d}(\hat \psi)=\Pi_E(\phi)$; then   we take a sequence of strongly admissible $C^2(\CC^d)$ weights $\{\phi_j\}$ with $\phi_j \uparrow \hat \psi$. We can apply Case 2 to $(\CC^d, \phi_j)$ and $(\CC^d,\hat \phi)$ to conclude
$$\lim_{n\to \infty}\frac{-(d+1)n_d}{2nd_n} [\mathcal B^{\infty}(\CC^d,n \phi_j):\mathcal B^{\infty}(\CC^d,n\hat \phi)]= \mathcal E(\Pi_{\CC^d}(\phi_j), \Pi_{\CC^d}(\hat \phi)).$$
But $\phi_j \uparrow \hat \psi$ implies $\Pi_{\CC^d}(\phi_j)\uparrow \Pi_{\CC^d}(\hat \psi)=\Pi_E(\phi)$ and hence 
\begin{equation}
\label{ephij}
\mathcal E(\Pi_{\CC^d}(\phi_j), \Pi_{\CC^d}(\hat \phi)) \ \hbox{converges to} \ \mathcal E(\Pi_E(\phi), \Pi_{\CC^d}(\hat \phi))
\end{equation}
as $j\to \infty$ by Lemma \ref{BT63}. 

We want to conclude that
\begin{equation}
\label{ephin}
\lim_{n\to \infty} \frac{-(d+1)n_d}{2nd_n} [\mathcal B^{\infty}(E,n \phi):\mathcal B^{\infty}(\CC^d,n\hat \phi)]= \mathcal E(\Pi_E(\phi), \Pi_{\CC^d}(\hat \phi)).
\end{equation}
To this end, first observe that 
$$-\mathcal E(\Pi_{\CC^d}(\phi_j), \Pi_{\CC^d}(\hat \phi))=\lim_{n\to \infty} \frac{(d+1)n_d}{2nd_n} [\mathcal B^{\infty}(\CC^d,n \phi_j):\mathcal B^{\infty}(\CC^d,n\hat \phi)]$$
$$\leq \liminf_{n\to \infty} \frac{(d+1)n_d}{2nd_n}[\mathcal B^{\infty}(E,n \phi):\mathcal B^{\infty}(\CC^d,n\hat \phi)]$$
$$\leq \limsup_{n\to \infty} \frac{(d+1)n_d}{2nd_n}[\mathcal B^{\infty}(E,n \phi):\mathcal B^{\infty}(\CC^d,n\hat \phi)]$$
since $\Pi_{\CC^d}(\phi_j)\uparrow \Pi_E(\phi)$ implies from (\ref{bvmon}) that 
$$[\mathcal B^{\infty}(\CC^d,n \phi_j):\mathcal B^{\infty}(\CC^d,n\hat \phi)] \leq [\mathcal B^{\infty}(E,n \phi):\mathcal B^{\infty}(\CC^d,n\hat \phi)].$$
Now we take a sequence of smooth, strongly admissible weights $\{\psi_j\}$ on $\CC^d$ with $\psi_j \downarrow \Pi_E(\phi)$; e.g., we may take $\psi_j =(1+\epsilon_j)[(\Pi_E(\phi))_{\epsilon_j}]$ where $(\Pi_E(\phi))_{\epsilon_j}$ is a smoothing of $\Pi_E(\phi)$. Then $\Pi_{\CC^d}(\psi_j)\downarrow \Pi_E(\phi)$ and
$$\limsup_{n\to \infty} \frac{(d+1)n_d}{2nd_n} [\mathcal B^{\infty}(E,n \phi):\mathcal B^{\infty}(\CC^d,n\hat \phi)]$$
$$\leq \lim_{n\to \infty} \frac{(d+1)n_d}{2nd_n} [\mathcal B^{\infty}(\CC^d,n \psi_j):\mathcal B^{\infty}(\CC^d,n\hat \phi)]$$
again by (\ref{bvmon}); this limit equals
$$-\mathcal E(\Pi_{\CC^d}(\psi_j), \Pi_{\CC^d}(\hat \phi))$$
by applying Case 2, this time to $(\CC^d, \psi_j)$ and $(\CC^d,\hat \phi)$. Now 
\begin{equation}
\label{epsi}
\mathcal E(\Pi_{\CC^d}(\psi_j), \Pi_{\CC^d}(\hat \phi)) \ \hbox{converges to} \ \mathcal E(\Pi_E(\phi), \Pi_{\CC^d}(\hat \phi))
\end{equation}
as $j\to \infty$ by (\ref{downlemma}). Then (\ref{ephij}) and (\ref{epsi}) imply (\ref{ephin}) which completes the proof of Theorem \ref{keycn}.

\section{Asymptotic weighted $P-$Fekete measures, weighted $P-$optimal measures and Bergman asymptotics.}\label{sec:fob}

As in \cite{BBN}, we will apply the following calculus lemma (cf., Lemma 7.6 in \cite{[BB]} or Lemma 3.1 in \cite{BBN}) to an appropriate sequence of real-valued functions $\{f_n\}$ in order to prove a general result, Proposition \ref{genres}, on convergence to the Monge-Amp\`ere measure of a weighted $P-$extremal function. This proposition utilizes the differentiability result, Proposition \ref{diffpropcor}, and yields immediate corollaries on the items in the title of this section. 

\begin{lemma}\label{calc}
Let $f_n$ be a sequence of real-valued, concave functions on $\RR$ and let $g$ be a function on $\RR$. Suppose
$$\liminf_{n\to \infty} f_n(t) \geq g(t) \ \hbox{for all} \ t \ \hbox{and} \ \lim_{n\to \infty} f_n(0) = g(0)$$
and that $f_n$ and $g$ are differentiable at $0$. Then $\lim_{n\to \infty}  f_n'(0) = g'(0)$.
\end{lemma}

\noindent Here ``differentiable at the origin'' means that the usual (two-sided) limit of the difference quotients exists; the conclusion is not true with one-sided limits.

As in Lemma \ref{1stderiv} in subsection \ref{gramsec}, given a closed set $E$, an admissible weight $w=e^{-Q}$ on $E$, and a function $u\in C(E)$, we consider the weight $w_t(z):=w(z)\exp(-tu(z)),$ $t\in\RR,$ and we let $\{\mu_n\}$ be a sequence of measures on $E$. 

{\it For the rest of this section, we take $E=K$, a compact set, so each $w_t$ is admissible. In addition, in computing Gram matrices, we fix the standard monomial basis $\beta_n=\{e_1,...,e_{d_n}\}$ of $Poly(nP)$; and we fix $v=H_P$ in the second slot of $\mathcal E(u,v)$.} 

Now let $\mu$ be a probability measure on $K$ and let $u\in C^2(K)$. Recalling (\ref{key}), define 
$$g(t):= - \log \delta^{w_t}(K)=\frac{1}{n_d d\mathcal A}\mathcal E(\Pi(Q+tu)).$$
Then
$$g(0)=- \log \delta^{w}(K)=\frac{1}{n_d d\mathcal A}\mathcal E(\Pi(Q)).$$
From Proposition \ref{diffpropcor} 
$$g'(0)= \frac{d+1}{n_d d\mathcal A}\int_Ku(z)(dd^c\Pi(Q))^d.$$

Note that for each $n$, $\mu_n$ is a candidate to be a $P-$optimal measure of order $n$ for $K$ and $w_t$. Thus, if $\mu_n^t$ is a $P-$optimal measure of order $n$ for $K$ and $w_t$, we have 
$$\det G_n^{\mu_n,w_t} \leq \det G_n^{\mu_n^t,w_t}$$
and, from Proposition \ref{tfd} (see Remark \ref{rmk42}),
$$\lim_{n\to \infty}  \frac{1}{2l_n} \cdot \log \det G_n^{\mu_n^t,w_t}= \log \delta^{w_t}(K)=-g(t).$$
Thus with 
$$f_n(t):=-{1\over 2 l_n}\log\,{\rm det}(G_n^{\mu_n,w_t})$$
as in (\ref{fn}), we have
$$f_n(0)=\frac{-1}{2l_n}\log\,{\rm det}(G_n^{\mu_n,w}) \ \hbox{and} \ \liminf f_n(t) \geq g(t) \ \hbox{for all} \ t.$$
From Lemma \ref{1stderiv}, we have
$$f_n'(0)={n\over l_n}\int_K u(z)B_n^{\mu_n,w}(z)d\mu_n$$
and from Lemma \ref{2ndderiv}, the functions $f_n(t)$ are concave, i.e., $f_n''(t)\le 0$.

Using Lemma \ref{calc} and (\ref{key}), we have the following general result.

\begin{proposition}
\label{genres}
Let $K\subset \CC^d$ be compact with admissible weight $w$. Let $\{\mu_n\}$ be a sequence of probability measures on $K$ with the property that 
\begin{equation}\label{fnhyp}
\lim_{n\to \infty}{1 \over 2 l_n}\log\,{\rm det}(G_n^{\mu_n,w})=\log \mathcal F_P(K,Q)
\end{equation}
i.e., $\lim_{n\to \infty}f_n(0)=g(0)$. Then 
$$\frac{n}{l_n}B_n^{\mu_n,w} d\mu_n \to \frac{d+1}{n_d d\mathcal A} (dd^c\Pi(Q))^d \ \hbox{weak-*; i.e.,}$$
\begin{equation}\label{strongasym}
\frac{n_d}{d_n}B_n^{\mu_n,w} d\mu_n \to  (dd^c\Pi(Q))^d \ \hbox{weak-}*.
\end{equation}
\end{proposition}

\noindent Note that since all $\mu_n$ are probability measures on $K$, to verify weak-* convergence, it suffices to test with $C^2-$functions on $K$. 

From Theorem \ref{keycn} (more precisely, Remark \ref{rmk42} and equation (\ref{wbmfinal})) we have the general Bergman asymptotic result.

\begin{corollary} \label{strongasymp}
{\bf [Bergman Asymptotics]} If $\mu$ is a Bernstein-Markov measure for the triple $(P,K,Q)$, then 
$$\frac{n_d}{d_n}B_n^{\mu,w} d\mu \to (dd^c\Pi(Q))^d \ \hbox{weak-}*.
$$
\end{corollary}

Next, suppose $\mu_n$ is a $P-$optimal measure of order $n$ for $K$ and $w$. 

\begin{corollary} \label{awom} 
{\bf [Weighted Optimal Measures]} 
Let $K\subset \CC^d$ be compact with admissible weight $w$. Let $\{\mu_n\}$ be a sequence of $P-$optimal measures for $K,w$. Then
$$\mu_n  \to \frac{1}{n_d}(dd^c\Pi(Q))^d \ \hbox{weak-}*.
$$
\end{corollary}
\begin{proof} We have $B_n^{\mu_n,w}=d_n$ a.e. $\mu_n$ on $K$ from \ref{wb} so that the result follows immediately from Proposition \ref{tfd} and Proposition \ref{genres}, specifically, equation (\ref{strongasym}).  \end{proof}

Finally, we prove the result promised in Section \ref{sec:back}.

\begin{corollary} \label{asympwtdfek} {\bf [Asymptotic Weighted $P-$Fekete Points]} Let $K\subset \CC^d$ be compact with admissible weight $w$. For each $n$, take points $z_1^{(n)},z_2^{(n)},\cdots,z_{d_n}^{(n)}\in K$ for which 
\begin{equation}\label{wam}
 \lim_{n\to \infty}\bigl[|VDM(z_1^{(n)},\cdots,z_{d_n}^{(n)})|w(z_1^{(n)})^n\cdots w(z_{d_n}^{(n)})^n\bigr]^{1 \over  l_n}=\mathcal F_P(K,Q)
\end{equation}
({\it asymptotically} weighted $P-$Fekete points) and let $\mu_n:= \frac{1}{d_n}\sum_{j=1}^{d_n} \delta_{z_j^{(n)}}$. Then
$$
\mu_n \to \frac{1}{n_d}(dd^c\Pi(Q))^d \ \hbox{weak}-*.
$$
\end{corollary}
\begin{proof} By direct calculation, we have $B_n^{\mu_n,w}(z_j^{(n)})=d_n$ for $j=1,...,d_n$ and hence a.e. $\mu_n$ on $K$. Indeed, this property holds for {\it any} discrete, equally weighted measure $\mu_n:= \frac{1}{d_n}\sum_{j=1}^{d_n} \delta_{z_j^{(n)}}$ with $$|VDM(z_1^{(n)},\cdots,z_{d_n}^{(n)})|w(z_1^{(n)})^n\cdots w(z_{d_n}^{(n)})^n\not =0.$$ 
Using 
$${\rm det}(G_n^{\mu_n,w})={1\over d_n^{d_n}}|VDM(z_1^{(n)},\cdots,z_{d_n}^{(n)})|^2w(z_1^{(n)})^{2n}\cdots w(z_{d_n}^{(n)})^{2n},$$
the result follows from Proposition \ref{genres}, specifically, equation (\ref{strongasym}). \end{proof}

\end{document}